\documentclass[a4paper, 12pt, onecolumn]{article}
\usepackage{geometry}
\usepackage{amsmath}
\usepackage{amsthm}
\usepackage{amsfonts}
\usepackage{bbm}
\usepackage{CJK}
\usepackage{fancyhdr}
\usepackage{graphicx}
\usepackage{psfrag}
\usepackage{amsfonts,amsmath,amsthm, amssymb}
\usepackage{latexsym, euscript, epic, eepic}
\usepackage{time}
\usepackage{txfonts}
\usepackage{colortbl}
\usepackage{stmaryrd}
\usepackage{mathrsfs}
\usepackage{txfonts}
\usepackage{amsfonts}
\usepackage{color}
\usepackage{lineno}
\usepackage[square, comma, sort&compress, numbers]{natbib}

\usepackage{indentfirst,latexsym,bm}

 \numberwithin{equation}{section}

\begin{document}
\newtheorem{theorem}{Theorem}[section]
\newtheorem{proposition}[theorem]{Proposition}
\newtheorem{remark}[theorem]{Remark}
\newtheorem{corollary}[theorem]{Corollary}
\newtheorem{definition}{Definition}[section]
\newtheorem{lemma}[theorem]{Lemma}
\newcommand{\wuhao}{\fontsize{10pt}{10pt}\selectfont}

\title{Sharp Estimates of the Generalized Euler-Mascheroni Constant
 \footnotetext{E-mail address:
htiren@zstu.edu.cn, }}
\author{\normalsize  Ti-Ren Huang$^1$\footnote{Corresponding author}, Bo-Wen Han$^1$, You-Ling Liu$^2$, Xiao-Yan Ma$^1$ \\
\scriptsize 1 Department of Mathematics, Zhejiang Sci-Tech
University, Hangzhou 310018, China\\
\scriptsize 2 Tourism College of Zhejiang, Hangzhou  311231, China }
\date{}
\maketitle
 \fontsize{12}{22}\selectfont\small
 \paragraph{Abstract:} Let $a\in (0, \infty)$, $\gamma(a)$ be the Generalized Euler-Mascheroni Constant, and  let
 \begin{align*}
 &x_n=\frac1a+\frac{1}{a+1}+\cdots+\frac{1}{a+n-1}-\ln\frac{a+n}{a},\\
 &y_n=\frac1a+\frac{1}{a+1}+\cdots+\frac{1}{a+n-1}-\ln\frac{a+n-1}{a}.
 \end{align*}
 In this paper, we determine the best possible constants $\alpha_i, \beta_i (i=1,2,3,4)$ such that the following inequalities
 \begin{align*}
 \frac{1}{2(n+a)-\alpha_1}\leq &\gamma(a)-x_n< \frac{1}{2(n+a)-\beta_1},\\
  \frac{1}{2(n+a)-\alpha_2}\leq &y_n-\gamma(a)< \frac{1}{2(n+a)-\beta_2},\\
  \frac{1}{2(n+a)}+\frac{\alpha_3}{(n+a)^2}\leq &\gamma(a)-x_n<\frac{1}{2(n+a)}+\frac{\beta_3}{(n+a)^2},\\
  \frac{1}{2(n+a-1)}+\frac{\alpha_4}{(n+a-1)^2}< &y_n-\gamma(a)\leq\frac{1}{2(n+a-1)}+\frac{\beta_4}{(n+a-1)^2}.
 \end{align*}
 are valid for all integers $n\geq 1$.
\\[10pt]
\\[10pt]
\section{\normalsize Introduction}\label{sec:bd}

One of the most important sequences in analysis and number theory of the form
\begin{align}\label{Eular r_n}
\gamma_n= 1+\frac12+\frac13+\cdots+\frac1n-\ln n,
\end{align}
considered by Leonhard Euler in 1735, is known to converge towards the limit
\begin{align}\label{Eular r_n}
\gamma=0.57721566490115328\cdots,
\end{align}
which is now called the Euler-Mascheroni Constant. For $\gamma_n-\gamma$,  many lower and upper estimates have
been given in the literature\cite{Alzer1,Chen,Mortici9,Qiu S L,Toth}. We remind some of them:

In \cite{Alzer1}, Alzer proved that for $n\geq 1$,
\begin{align*}
\frac{1}{2n+1}\leq \gamma_n-\gamma \leq \frac{1}{2n}.
\end{align*}

T$\acute{o}$th \cite{Toth} proved that for $n\geq 1$,
\begin{align}
\frac{1}{2n+\frac25}< \gamma_n-\gamma \leq \frac{1}{2n+\frac13}.
\end{align}

In \cite{Chen}, Chen proved that for $n\geq 1$,
\begin{align}\label{chen}
%\frac{1}{2n+\frac{2\gamma-1}{1-\gamma}}\leq \gamma_n-\gamma \leq \frac{1}{2n+\frac13},
\frac{1}{2n+\alpha}\leq \gamma_n-\gamma < \frac{1}{2n+\beta},
\end{align}
where the constants $\alpha=(2\gamma-1)/(1-\gamma)$ and $\beta=1/3$  are the best possible.

Qiu and Vuorinen \cite{Qiu S L} showed the double inequality
\begin{align}\label{Qiu}
%\frac{1}{2n}-\frac{1}{2n^2}\leq \gamma_n-\gamma \leq \frac{1}{2n}-\frac{\gamma-1/2}{n^2}, \quad n\geq 1.
\frac{1}{2n}-\frac{\alpha}{n^2}< \gamma_n-\gamma \leq \frac{1}{2n}-\frac{\beta}{n^2}, \quad n\geq 1,
\end{align}
where the constants $\alpha=1/12$ and $\beta=\gamma-1/2$  are the best possible.

For every $a > 0$; the numbers of the form
\begin{align}\label{gamma(a)}
\gamma(a)=\lim_{n\rightarrow \infty}\left(\frac1a+\frac{1}{a+1}+\cdots+\frac{1}{a+n-1}-\ln\frac{a+n-1}{a}\right)
\end{align}
were introduced in the monograph by Knopp \cite{Knopp}. There are known now as the
generalized Euler-Mascheroni constant, since $\gamma(1)=\gamma$. Recently,  the
generalized Euler-Mascheroni constant $\gamma(a)$ has been the
subject of intensive research \cite{S,Mortici9,Berinde}, similar to $\gamma$.

In \cite{S}, S$\hat{\i}$nt$\check{a}$m$\check{a}$rian consider the following sequences

\begin{align}
 &x_n=\frac1a+\frac{1}{a+1}+\cdots+\frac{1}{a+n-1}-\ln\frac{a+n}{a},\label{x_n}\\
 &y_n=\frac1a+\frac{1}{a+1}+\cdots+\frac{1}{a+n-1}-\ln\frac{a+n-1}{a}.\label{y_n}
 \end{align}
 and proved that for $n\geq 1$, the following inequalities hold,
 \begin{align}
 \frac{1}{2(n+a)}\leq \gamma(a)-x_n\leq \frac{1}{2(n+a-1)}, \quad \frac{1}{2(n+a)}\leq y_n-\gamma(a)\leq \frac{1}{2(n+a-1)}.
 \end{align}
 Hence, we can easily know that $x_n,y_n$ converge to $\gamma(a)$ like $n^{-1}$.

In \cite{Berinde}, Berinde and Mortici  gave better bounds for $\gamma(a)-x_n, y_n-\gamma(a)$, showing the following
\begin{theorem}\label{theorem1}
For each $a>0,n\geq 2$, then
 \begin{align}
 \frac{1}{2(n+a)-\frac14}&< \gamma(a)-x_n< \frac{1}{2(n+a)-\frac13},\quad
 \frac{1}{2(n+a)-\frac43}&< y_n-\gamma(a)<\frac{1}{2(n+a)-\frac53}.
 \end{align}
 \end{theorem}
In the same paper, Berinde and Mortici   obtained the following theorem.
\begin{theorem}
a) For each $a\geq \frac{13}{30}$ and each integer $n\geq 1$, then
 \begin{align}
 \frac{1}{2(n+a)-\frac13+\frac{1}{18n}}&\leq \gamma(a)-x_n.
 \end{align}
b) For each $a\geq \frac{17}{30}$ and each integer$n\geq 1$, then
\begin{align}
 \frac{1}{2(n+a)-\frac53+\frac{1}{18n}}&\leq y_n-\gamma(a).
 \end{align}
 \end{theorem}

It is natural to extend the above inequalities (\ref{chen}) and (\ref{Qiu}) in terms of generalized Euler-Mascheroni constant $\gamma(a)$. In this paper, we will consider  the two sequences $\gamma-x_n, y_n-\gamma$ where $x_n,y_n$ are defined by  (\ref{x_n}) and (\ref{y_n})  respectively.

The main results are stated as follows.
\begin{theorem}\label{main theorem1}
For each $a>0$ and integer $n\geq 1$, let  $\gamma(a)$ be generalized Euler-Mascheroni constant.

(1) Let the sequences $x_n$ be defined by (\ref{x_n}), then
\begin{align}\label{inquealities1}
\frac{1}{2(n+a)-\alpha_1}&\leq \gamma(a)-x_n<\frac{1}{2(n+a)-\beta_1},
\end{align}
with the best possible constants
\begin{align}
\alpha_1=2(1+a)-\frac{1}{\psi(1+a)-\ln(1+a)},\quad \beta_1=\frac13.
\end{align}

(2) Let the sequences $y_n$ be defined by (\ref{y_n}), then

\begin{align}\label{inquealities2}
\frac{1}{2(n+a)-\alpha_2}&\leq y_n-\gamma(a)<\frac{1}{2(n+a)-\beta_2},
\end{align}
with the best possible constants
\begin{align}
\alpha_2=2(1-d),\qquad  \beta_2=\frac53,
\end{align}
where
\begin{align*}
d=\max\{\tilde{f}_2(a),\tilde{f}_2(1+a),\tilde{f}_2(2+a)\}, \quad \tilde{f}_2(x)=\frac{1}{2(\psi(x+1)-\ln(x))}-x.
\end{align*}
\end{theorem}

\begin{theorem}\label{main theorem2}
For each $a>0$, and integer $n\geq 1$, let the sequencess $x_n,y_n$ be defined by (\ref{x_n}), (\ref{y_n}) and  $\gamma(a)$ be generalized Euler-Mascheroni constant, then
\begin{align}
\frac{1}{2(n+a)}+\frac{\alpha_3}{(n+a)^2}&\leq \gamma(a)-x_n<\frac{1}{2(n+a)}+\frac{\beta_3}{(n+a)^2},\label{inquealities3}\\
\frac{1}{2(n+a-1)}+\frac{\alpha_4}{(n+a-1)^2}&<y_n- \gamma(a)\leq\frac{1}{2(n+a-1)}+\frac{\beta_4}{(n+a-1)^2},\label{inquealities4}
\end{align}
with the best possible constants
\begin{align}
\alpha_3&=(1+a)^2 [\ln(1+a)-\psi(1+a)]-\frac{1+a}{2},\quad \beta_3=\frac{1}{12},\\
\alpha_4&=-\frac{1}{12},\quad\qquad \qquad \qquad  \beta_4=a^2 [\psi(a)-\ln(a)]+\frac{a}{2}.
\end{align}
\end{theorem}

\section{\normalsize Preliminaries}\label{sec:bd}
In this section, we give out several formulas and lemmas in order to establish our main results stated in section 1. Firstly, let us recall some known
results for the psi (or digamma) function $\psi(x)$.

 For real numbers $x, y>0$, the gamma and psi functions are
defined as
\begin{align*}
\Gamma(x)=\int_0^\infty t^{x-1}e^{-t}dt,
\quad \psi(x)=\frac{\Gamma'(x)}{\Gamma(x)},
\end{align*}
respectively.

The  psi function $\psi(x)$  has the following Recurrence Formulas \cite{Abramowitz M}
\begin{align}
\psi(n+z)=\frac{1}{(n-1)+z}+\frac{1}{(n-2)+z}+\cdots+\frac{1}{2+z}+\frac{1}{1+z}+\frac1z+\psi(z).\label{Recurrence Formulas2}
\end{align}
 and the Asymptotic Formulas (\cite{Abramowitz M}),
\begin{align}\label{Asymptotic Formulas}
\psi(z)\sim \ln(z)-\frac{1}{2z}-\frac{1}{12z^2}+\frac{1}{120z^4}-\frac{1}{252z^6}+\cdots \quad (z\rightarrow \infty \quad\hbox{in}\quad |\arg z|<\pi),
\end{align}
According to (\ref{Recurrence Formulas2}) and the definition of $x_n, y_n$ , we have
\begin{align}
&x_n=\psi(n+a)-\psi(a)-\ln\frac{n+a}{a},\label{x_n}\\
&y_n=\psi(n+a)-\psi(a)-\ln\frac{n+a-1}{a}.\label{y_n}
\end{align}
By the definition of  $\gamma(a)$ (\ref{gamma(a)}) and the Asymptotic Formulas (\ref{Asymptotic Formulas}), then
\begin{align}
\gamma(a)=\lim_{n\rightarrow\infty} y_n=\lim_{n\rightarrow\infty}\left(\psi(n+a)-\ln(n+a-1)+\ln(a)-\psi(a)\right)=\ln(a)-\psi(a).
\end{align}
Hence
\begin{align}
&\gamma(a)-x_n=\ln(n+a)-\psi(n+a),\label{gamma-x_n}\\
&y_n-\gamma(a)=\psi(n+a)-\ln(n+a-1).\label{y_n-gamma}
\end{align}

\begin{lemma}\label{lemma1}
(1) The function
\begin{align}
f_1(x)=\frac{1}{\ln(x)-\psi(x)}-2x
\end{align}
is strictly decreasing from $(1,\infty)$ onto  $(-1/3,1/\gamma-2)$.\\
(2) The function
\begin{align}
f_2(x)=\frac{1}{\psi(x+1)-\ln(x)}-2x
\end{align}
is strictly decreasing from $[2,\infty)$ onto  $(1/3,f_2(2)]$.
\end{lemma}
\begin{proof}
(1) Differentiation gives
\begin{align*}
\left(\ln(x)-\psi(x)\right)^2f_1'(x)=\psi'(x)-\frac1x-2(\ln(x)-\psi(x))^2,
\end{align*}
Using the inequalities \cite{Chen2}
\begin{align*}
\psi'(x)-\frac1x<\frac{1}{2x^2}+\frac{1}{6x^3}-\frac{1}{30x^5}+\frac{1}{42x^7},\quad
\ln(x)-\psi(x)>\frac{1}{2x}+\frac{1}{12x^2}-\frac{1}{120x^4}.
\end{align*}
We have
\begin{align}
\left(\ln(x)-\psi(x)\right)^2f_1'(x)<\frac{1}{50400 x^8}F_1(x),
\end{align}
where 
\begin{align}
F_1(x)=-207-3840(x-1)-6580(x-1)^2-3640(x-1)^3-700(x-1)^4,
\end{align}
we have $F(x)<0$, $f_1'(x)<0$ for $x\geq 1$, and the monotonicity  of $f_1(x)$
follows.

Clearly, $f_1(1)=1/\gamma-2$. The limiting value $\lim_{x\rightarrow \infty}f_1(x)=-1/3$ follows from
 the  Asymptotic Formulas (\ref{Asymptotic Formulas}).

(2) Differentiation yields
\begin{align*}
2\left(\psi(x+1)-\ln(x)\right)^2f_2'(x)=\frac1x+\frac{1}{x^2}-\psi'(x)-2(\psi(x)+\frac1x-\ln(x))^2.
\end{align*}
Using the inequalities \cite{Chen2}, for $x>0$,
\begin{align*}
\frac1x+\frac{1}{x^2}-\psi'(x)<\frac{1}{2x^2}-\frac{1}{6x^3}+\frac{1}{30x^5},\quad
\psi(x)+\frac1x-\ln(x)>\frac{1}{2x}-\frac{1}{12x^2}+\frac{1}{120x^4}-\frac{1}{252x^6},
\end{align*}
 we obtain
\begin{align*}
2\left(\psi(x+1)-\ln(x)\right)^2f_2'(x)<-\frac{F_2(x)}{3175200x^{12}}
\end{align*}
where
\begin{align*}
F_2(x)=&3217636+17887632(x-2)+39443124(x-2)^2+47009928(x-2)^3
+33797841(x-2)^4\\
&+15180480(x-2)^5+4189500(x-2)^6+652680(x-2)^7+44100(x-2)^8\\
>&0\quad (x\geq 2).
\end{align*}
Hence $f_2(x)$ is a decreasing function on $[2,\infty)$.
The limiting value $\lim_{x\rightarrow \infty}f_2(x)=1/6$ follows from
 the  Asymptotic Formulas (\ref{Asymptotic Formulas}).
 \end{proof}

The following Lemma follows from theorem 1.7 in \cite{Qiu S L}.

 \begin{lemma}\label{lemma2}
 The function
 \begin{align}
f_3(x)=x^2 (\psi(x)-\ln(x))+\frac x2
\end{align}
 is strictly decreasing and convex
from  $(0,\infty)$ onto $(-1/12,0)$.
\end{lemma}

\section{\normalsize Proof of the main theorem}\label{sec:bd}

\emph{\textbf{Proof of  Theorem }\ref{main theorem1}}. \quad  (1).  According to (\ref{gamma-x_n}),
the inequality (\ref{inquealities1}) can be written as
 \begin{align*}
-\beta<\frac{1}{\ln(n+a)-\psi(n+a)}-2(n+a)<-\alpha.
 \end{align*}
 By the Lemma \ref{lemma1} (1), we know that the sequence
 \begin{align*}
 f_1(n+a)=\frac{1}{\ln(n+a)-\psi(n+a)}-2(n+a), \quad (n\in \mathbb{N})
 \end{align*}
 is strictly decreasing. This leads to
 \begin{align*}
 -\frac13=\lim_{n\rightarrow \infty} f_1(n)<f_1(n)\leq f_1(1)=\frac{1}{\ln(1+a)-\psi(1+a)}-2(1+a)
 \end{align*}
 Hence the best possible constants are
\begin{align*}
\alpha_1=2(1+a)-\frac{1}{\psi(1+a)-\ln(1+a)},\quad \beta_1=\frac13.
\end{align*}

(2).  According to (\ref{y_n-gamma}), the inequality (\ref{inquealities2}) can be written as
\begin{align*}
1-\frac\beta2<\frac{1}{2 (\psi(n+a)-\ln(n+a-1))}-(n+a-1)\leq1-\frac\alpha2.
\end{align*}
 By the Lemma \ref{lemma1} (2),  the sequence
\begin{align*}
 \tilde{f}_2(n+a-1)=\frac{1}{2 (\psi(n+a)-\ln(n+a-1))}-(n+a-1) \quad (n\geq 2)
 \end{align*}
 is strictly decreasing. Hence
 \begin{align*}
 \frac16=\lim_{n\rightarrow \infty} \tilde{f}_2(n)<\tilde{f}_2(n)\leq \max\{\tilde{f}_2(a),\tilde{f}_2(1+a),\tilde{f}_2(2+a)\},
 \end{align*}
 let $d:=\max\{\tilde{f}_2(a),\tilde{f}_2(1+a),\tilde{f}_2(2+a)\}$, hence the best possible constants are
\begin{align}
\alpha_2=2(1-d),\qquad  \beta_2=\frac53.
\end{align}

 \emph{\textbf{Proof of  Theorem} \ref{main theorem2}}.
 According to (\ref{gamma-x_n})-(\ref{y_n-gamma}), the inequality (\ref{inquealities3})-(\ref{inquealities4}) can be written as
 \begin{align*}
 &\alpha_3\leq (n+a)^2\left(\ln(n+a)-\psi(n+a)\right)-\frac{(n+a)}{2}<\beta_3,\\
  &\alpha_4<(n+a-1)^2\left(\psi(n+a-1)-\ln(n+a-1)\right)+\frac{(n+a-1)}{2}\leq\beta_4.
 \end{align*}
 By the Lemma \ref{lemma2}, we know that the sequence
 \begin{align*}
 \tilde{f}_3(n+a-1)=(n+a-1)^2\left(\psi(n+a-1)-\ln(n+a-1)\right)+\frac{(n+a-1)}{2}, \quad (n\in \mathbb{N})
 \end{align*}
  is strictly decreasing, and $\lim_{n\rightarrow \infty} f_3(n)=-1/12$. Hence, the best possible constants are
\begin{align*}
\alpha_3&=(1+a)^2 [\ln(1+a)-\psi(1+a)]-\frac{1+a}{2},\quad \beta_3=\frac{1}{12},\\
\alpha_4&=-\frac{1}{12},\quad \quad\quad\quad\quad\quad\quad\beta_4=(a)^2 [\psi(a)-\ln(a)]+\frac{a}{2}.
\end{align*}

\begin{remark}
(1).
Taking $a=1$ in Theorem \ref{main theorem1} (2), then we get  the inequality (\ref{chen}) with the constants $\alpha=(2\gamma-1)/(1-\gamma)$ and $\beta=1/3$  are the best possible.

(2).Taking $a=1$ in the inequality(\ref{inquealities4}) of Theorem \ref{main theorem2},  then we get  the inequality (\ref{Qiu}) with the constants $\alpha=1/12$ and $\beta=\gamma-1/2$  are the best possible.

(3). From Theorem \ref{main theorem1}, we know that the constants 1/3 and 5/3 are best possible in Theorem  \ref{theorem1} for any  $a>0$.
\end{remark}

\noindent\textbf{Acknowledgements} This work was completed with the support of National Natural Science Foundation of China (No.11401531, No.11601485.), the Science Foundation of Zhejiang Sci-Tech University (ZSTU)(No.14062093-Y) and the Natural Science Foundation of Zhejiang Province (No.LQ17A010010)

\clearpage


\begin{thebibliography}{99}
\wuhao
\setlength{\baselineskip}{1em}
\setlength{\parskip}{-0.5em}
\bibitem{Abramowitz M}
M. Abramowitz, A. I Stegun, Handbook of Mathematical Functions, with Formulas, Graphs, and Mathematical Tables.
\textit{New York: Dover Publications,} 1966.

\bibitem{Alzer1}
H. Alzer, Inequalities for the gamma and polygamma functions, \textit{Abh. Math. Sem. Univ. Hamb.} 68 (1998) 363-372.

%\bibitem{Alzer2}
%H. Alzer, S. Koumandos, Series representations for $\gamma$ and other mathematical constants, \textit{Anal. Math.} 34 (1) (2008) 1-8.


%\bibitem{Bailey}
%D.H. Bailey, Numerical results on the transcendence of constants involving $\pi,e $ and Euler's constant, \textit{Math. Comput.} 50 (1988) 275-281.

\bibitem{Berinde}
 V. Berinde and C. Mortici, New sharp estimates of the generalized Euler Mascheroni
constant, \textit{Math. Inequal. Appl.}16 (1) (2013), 279-288.
\bibitem{Chen}
 C.-P. Chen,  Inequalities for the Euler-Mascheroni constant, \textit{Appl. Math. Lett.} 23 (2010) 161-164.
 %\bibitem{ChenCP}
 %C.-P. Chen, Monotonicity properties of functions related to the psi function, \textit{Appl. Math. Comput.} 217 (2010) 2905-2911.

\bibitem{Chen2}
  C.-P. Chen, F. Qi, The best lower and upper bounds of harmonic sequence, \textit{RGMIA} 6 (2) (2003). Art. 14.



 %\bibitem{Chen3}
 %C.-P. Chen, H.M. Srivastava, New representations for the Lugo and Euler-Mascheroni constants, \textit{ Appl. Math. Lett.} 24 (7) (2011) 1239-1244.

%\bibitem{DeTemple}
% D.W. DeTemple, A quicker convergence to Euler's constant, \textit{Am. Math. Monthly} 100 (5) (1998) 468-470.


%\bibitem{Karatsuba}
%E.A. Karatsuba, On the computation of the Euler constant $\gamma$. Computational methods from rational approximation theory (Wilrijk, 1999), \textit{Numer. Algorithms} 24 (1-2) (2000) 83-97.

\bibitem{Knopp}
 K. Knopp, Theory and Applications of Infinite Series, vol. 453, Blackie, London, 1951.

%\bibitem {Mortici1}
%C. Mortici, A. Vernescu, An improvement of the convergence speed of the sequence $(\gamma)_n)_{n\geq 1}$ converging to Euler's constant, \textit{An. St. Univ. Ovidius Constanta.} 13 (1) (2005) 97-100.

% \bibitem{Mortici2}
%C. Mortici, A. Vernescu, Some new facts in discrete asymptotic analysis, \textit{Math. Balkanica, New Ser.} 21 (3-4) (2007) 301-308.

%\bibitem{Mortici3}
%C. Mortici, Product approximations via asymptotic integration, \textit{Am. Math. Monthly,} 117 (5) 434-441.


%\bibitem{Mortici4}
% C. Mortici, An ultimate extremely accurate formula for approximation of the factorial function, \textit{Arch. Math. (Basel)} 93 (1) (2009) 37-45.

%\bibitem{Mortici5}
%C. Mortici, New approximations of the gamma function in terms of the digamma function, \textit{Appl. Math. Lett.} 23 (1) (2010) 97-100.

%\bibitem{Mortici6}
 %C. Mortici, An ultimate extremely accurate formula for approximation of the factorial function, \textit{Arch. Math. (Basel)} 93 (1) (2009) 37-45.

%\bibitem{Mortici7}
% C. Mortici, New sharp bounds for gamma and digamma functions, \textit{An. Stiint. Univ. A.I. Cuza Iasi Ser. N. Matem.} 57 (1) (2011).

 %\bibitem{Mortici8}
 %C. Mortici, Complete monotonic functions associated with gamma function and applications, \textit{Carpathian J. Math.} 25 (2) (2009) 186-191.


  \bibitem{Mortici9}
   C. Mortici, Improved convergence towards generalized Euler-Mascheroni constant. \textit{Appl. Math. Comput.} 215 (9) (2010), 3443-3448.


 %\bibitem{Mortici10}
 %C. Mortici, Optimizing the rate of convergence in some new classes of sequences convergent to Euler's constant,\textit{ Anal. Appl. (Singap.)} 8 (1) 2010 99-107.


 %\bibitem{P}
 % G. P$\acute{o}$lya, G. Szeg$\ddot{o}$, Problems and Theorems in Analysis, vols. I and II, \textit{Springer-Verlag, Berlin, Heidelberg,} 1972.

  \bibitem{Qiu S L}
S.-L. Qiu, M. Vuorinen, Some properties of the gamma and psi functions, with applications, \textit{Math. Comput.} 74 (250) (2005) 723-742.

\bibitem{S}
 A. S$\hat{i}$nt$\breve{a}$m$\breve{a}$rian, A generalization of Euler's constant, \textit{Numer. Algorithms} 46 (2) (2007) 141-151.

% \bibitem {Sweeney}
%D.W. Sweeney, On the computation of Euler's constant, \textit{Math. Comput.} 17 (1963) 170-178.

\bibitem {Toth}
 L. T$\acute{o}$th, Problem E 3432, \textit{Am. Math. Monthly} 98 (3) (1991) 264.

%\bibitem {Vernescu}
%A. Vernescu, A new accelerate convergence to the constant of Euler, \textit{Gazeta Matem., Ser. A, Bucharest} XVII(XCVI) (4) (1999) 273-278.

%\bibitem {Young}
%R.M. Young, Euler's constant, \textit{Math. Gaz.} 75 (472) (1991) 187-190.



\end{thebibliography}
\end{document}